\documentclass{amsart}
\usepackage{amsfonts,amssymb,amsmath,amsthm}
\usepackage{url}
\usepackage{enumerate}

\makeatletter
\@namedef{subjclassname@2010}{  \textup{2010} Mathematics Subject Classification.}
\makeatother
\newtheorem{thrm}{Theorem}[section]

\newtheorem{prop}[thrm]{Proposition}

\theoremstyle{definition}
\newtheorem{definition}[thrm]{Definition}
\newtheorem{remark}[thrm]{Remark}
\numberwithin{equation}{section}
\email{hzoubeir2014@gmail.com}
\input{tcilatex}

\begin{document}
\address{ }
\author{Hicham Zoubeir}
\address{Ibn Tofail University, Department of Mathematics, Faculty of
Sciences, P.O.B : $133,$ Kenitra, Morocco.}
\title[Solvability of some nonlinear FFDE]{Solvability in Gevrey classes of
some nonlinear fractional functional differential equations}

\begin{abstract}
Our purpose in this paper is to prove, under some regularity conditions on
the datas, the solvability in a Gevrey class of bound $-1$\ on the interval $%
\left[ -1,1\right] $ of a class of nonlinear fractional functional
differential equations.
\end{abstract}

\dedicatory{$\emph{This}$ $\emph{modest}$ $\emph{work}$ $\emph{is}$ $\emph{%
dedicated}$ $\emph{to}$ $\emph{the}$ $\emph{memory}$ $\emph{of}$ $\emph{our}$
$\emph{beloved}$ $\emph{master}$ $\emph{Ahmed}$ $\emph{Intissar}$ $\emph{%
(1951-2017),}$ $\emph{a}$ $\emph{brilliant}$ $\emph{mathematician}$ $(\emph{%
PhD}$ $\emph{at}$ $\emph{M.I.T,}$ \emph{Cambridge}$\emph{),}\  \emph{a}$ $%
\emph{distinguished}$ $\emph{professor,}$ $\emph{a}$ $\emph{man}$ $\emph{with%
}$ $\emph{a}$ $\emph{golden}$ $\emph{heart.}$\emph{\ }}
\maketitle

\section{\textbf{Introduction}}

The fractional calculus has grown up from the speculations of early
mathematicians of the $17^{\text{th}}$ and $18^{\text{th }}$centuries like
G. W. Leibnitz, I. Newton, L. Euler, G. F. de L'Hospital, J. L. Lagrange (%
\cite{OLDH}). In the $19^{\text{th }}$century, other eminent mathematicians
like P. S. Laplace, J. Liouville, B. Riemann, E. A. Holmgren, O. Heaviside,
A. Grunwald, A. Letnikov, J. B. J. Fourier, N. H. Abel have used the ideas
of fractional calculus to solve some physical or mathematical problems (\cite%
{OLDH}). In the $20^{\text{th }}$century, several mathematicians (S.
Pincherle, O. Heaviside, G. H. Hardy, H. Weyl, E. Post, T. J. Fa Bromwich,
A. Zygmund, A. Erdelyi, R. G. Buschman, M. Caputo etc.) have made
considerable progress in their quest for rigor and generality, to build the
fractional calculus and its applications on rigorous and solid mathematical
foundations (\cite{OLDH}). Actually the fractional calculus allows the
mathematical modeling of social and natural phenomena in a more powerful way
than the classical calculus. Indeed fractional calculus has a lot of
applications in different areas of pure and applied sciences like
mathematics, physics, engeneering, fractal phenomena, biology, social
sciences, finance, economy, chemistry, anomalous diffusion, rheology (\cite%
{BAGL}-\cite{BARP}, \cite{CRAI}, \cite{DJOR}, \cite{JUMA}-\cite{PAKH}, \cite%
{ROS1}, \cite{SENG}, \cite{SHIM}, \cite{WANG}). It is then of capital
importance to develop for fractional calculus, the mathematical tools
analogous to those of classical calculus (\cite{BAL1}, \cite{BAL2}, \cite%
{OLDH}, \cite{ROS1}, \cite{ROS2}). The fractional differential equations (%
\cite{AGAR}, \cite{DIE1}, \cite{DIE2}, \cite{DING}, \cite{POD}, \cite{ROS2})
are a particularly important case of such fundamental tools. An important
type of fractional differential equations is that of fractional functional
differential equations (FFDE) (\cite{BENCH}, \cite{CHANG}, \cite{LAKSH}, 
\cite{SOUSA}) which are the fractional analogue to functional differential
equations (\cite{AZBE}, \cite{DRIV}, \cite{HALE}, \cite{MYSH}). FFDE enable
the study of some physical, biological, social, economical processes
(automatic control, financial dynamics, economical planning, population
dynamics, blood cell dynamics, infectious disease dynamics. etc) with
fractal memory and non-locality effects and where the rate of change of the
state of the systems depends not only on the present time but on other
different times which are functions of the present time (\cite{BHALE}, \cite%
{LIU}, \cite{PIMEN}). The question then arises of the choice of a suitable
framework for the study of the solvability of these equations. But since the
functional Gevrey spaces play an important role in various branches of
partial and ordinary differential equations (\cite{ARAU}, \cite{CAML}, \cite%
{HOLM}, \cite{TEMA}), we think that these functional spaces can play the
role of such convenient framework. However let us pointwise that in order to
make these spaces adequate to our specific setting, it is necessary to make
a modification to their definition. This leads us to the definition of a new
Gevrey classes namely the Gevrey classes $G_{l,q_{1}}(\left[ q_{1},q_{2}%
\right] )$ of bound $q_{1}$ and index $l>0$ on an interval $\left[
q_{1},q_{2}\right] .$ Our purpose in this paper is to prove, under some
regularity conditions on the datas, the solvability in a Gevrey class of the
form $G_{k,-1}(\left[ -1,1\right] )$ of a class of nonlinear FFDE. Our
approach is mainly based on a theorem that we have proved in (\cite{BEND}).
The notion of fractional calculus we are interested in is the Caputo
fractional calculus. Some examples are given to illustrate our main results.

\section{\textbf{Preliminary notes and statement of the main result}}

\subsection{Basic notations}

Let $F:E\longrightarrow E$ \ be a mapping from a nonempty set $E$ \ into
itself. $F^{\text{ }\left \langle n\right \rangle }(n\in \mathbb{N})$
denotes iterate of $F$ of order $n$ for the composition of mappings.

For $z\in \mathbb{C}$ and $h>0,$ $B(z,h)$ is the open ball in $\mathbb{%
C\simeq 
\mathbb{R}
}^{2}$ with the center $z$ and radius $h.$

Let $S_{1}$ and$\ S_{2}$ be two nonempty subsets of $\mathbb{C}$ such that $%
S_{1}\subset S_{2}$ and $f:S_{2}\rightarrow 
\mathbb{C}
$ a mapping. We denote by $f_{|S_{1}}$ the restriction of the mapping $f$ to
the set $S_{1}.$

For $z\in \mathbb{C}$ and $S\subset \mathbb{C}$ $(S$ nonempty) we set : 
\begin{equation*}
\rho (z,S):=\underset{\zeta \in S}{\inf }(|z-\zeta |)
\end{equation*}

For $l,$ $\varphi ,$ $r>0$ and $n\in \mathbb{N}^{\ast }$ we set for every
nontrivial compact interval $\left[ q_{1},q_{2}\right] $ of $%
\mathbb{R}
:$ 
\begin{equation*}
\left \{ 
\begin{array}{c}
\left[ q_{1},q_{2}\right] _{r}:=\{x+\zeta :x\in \left[ q_{1},q_{2}\right] ,%
\text{ }\zeta \in B(0,r)\} \\ 
\left[ q_{1},q_{2}\right] _{_{l,r,n}}:=\left[ q_{1},q_{2}\right] _{rn^{\frac{%
-1}{l}}} \\ 
\left[ q_{1},q_{2}\right] ^{\varphi ,r}:=\{q_{1}+se^{i\theta }:s\in
]0,q_{2}-q_{1}+r[,\text{ }\theta \in ]-\varphi ,\varphi \lbrack \} \\ 
\left[ q_{1},q_{2}\right] ^{r}:=\left[ q_{1},q_{2}\right] ^{r,r} \\ 
\left[ q_{1},q_{2}\right] ^{l,r,n}:=\left[ q_{1},q_{2}\right] ^{rn^{\frac{-1%
}{l}}}%
\end{array}%
\right.
\end{equation*}

Thus we have :%
\begin{equation*}
\left \{ 
\begin{array}{c}
\left[ q_{1},q_{2}\right] _{r}=\{z\in \mathbb{C}:\text{ }\varrho (z,\left[
q_{1},q_{2}\right] )<r\} \\ 
\left[ q_{1},q_{2}\right] _{l,r,n}=\{z\in \mathbb{C}:\text{ }\varrho (z,%
\left[ q_{1},q_{2}\right] )<rn^{\frac{-1}{l}}\}%
\end{array}%
\right.
\end{equation*}

\begin{remark}
The following inclusions hold for every $d\in ]q_{1},q_{2}],$ $r\in
]0,d-q_{1}[$ and $n\in \mathbb{N}^{\ast }:$%
\begin{equation}
\left \{ 
\begin{array}{c}
\lbrack d,s_{2}]_{r}\subset \left[ q_{1},q_{2}\right] ^{r} \\ 
\lbrack d,q_{2}]_{_{l,r,n}}\subset \left[ q_{1},q_{2}\right] ^{l,r,n}%
\end{array}%
\right.  \label{inclusion}
\end{equation}
\end{remark}

Let $f:S\longrightarrow 
\mathbb{C}
$ \ be a bounded function. $||f$ $||_{_{\infty ,S}}$ denotes the quantity :%
\begin{equation*}
||f\text{ }||_{_{\infty ,S}}:=\sup_{z\in S}|f(s)|
\end{equation*}

By $C^{0}([-1,1])$ $($resp. $C^{1}([-1,1]))$ we mean the set of all\ complex
valued defined and continuous $($resp. of class $C^{1})$ on the interval $%
[-1,1].$ $C^{0}([-1,1])$ endowed with the uniform norm :%
\begin{equation*}
\left \Vert \cdot \right \Vert _{\infty ,[-1,1]}:f\in C^{0}([-1,1])\mapsto
\left \Vert f\right \Vert _{\infty ,[-1,1]}
\end{equation*}%
becomes a Banach space. For every $r\geq 0,$ $\overline{\Delta }_{\infty
}(r) $ denotes then the closed ball, in this Banach space$,$ of center the
null function and radius $r.$

Let $\xi _{1},$ $\xi _{2}\in 
\mathbb{C}
.$ We denote by $\underrightarrow{\xi _{1},\xi _{2}}$ the linear path
joining $\xi _{1}$ to $\xi _{2}:$%
\begin{equation*}
\begin{array}{cccc}
\underrightarrow{\xi _{1},\xi _{2}}: & \left[ 0,1\right] & \rightarrow & 
\mathbb{C}
\\ 
& t & \mapsto & \left( 1-t\right) \xi _{1}+t\xi _{2}%
\end{array}%
\end{equation*}

Along this paper $k>0$ and $\alpha \in \left] 0,1\right[ $ are a fixed
numbers.

\subsection{Fractional derivatives and integrals}

\begin{definition}
Let $\delta \in \left] 0,1\right[ $ and $f$ \ a Lebesgue-integrable function
on the nontrivial compact interval $[q_{1},q_{2}].$ The Caputo fractional
integral of order $\delta $ and lower bound $q_{1}$ of the function $f$ (%
\cite{DIE1},\cite{DIE2},\cite{POD},\cite{ROS1},\cite{ROS2}) is the function
denoted by $^{c}I$ $_{a}^{\delta }f$ \ and defined by $:$%
\begin{equation*}
^{c}I\text{ }_{q_{1}}^{\delta }f\text{ }(t):=\frac{1}{\Gamma (\delta )}%
\int_{q_{1}}^{t}(t-s)^{\delta -1}f\text{ }(s)ds,\quad t\in \lbrack
q_{1},q_{2}]
\end{equation*}%
where $\Gamma $ denotes the classical Gamma function.
\end{definition}

\begin{remark}
If the function $f$ \ is continuous on the interval $[q_{1},q_{2}],$ then
the function $^{c}I$ $_{q_{1}}^{\delta }f$ \ is well defined and continuous
on the entire interval $[q_{1},q_{2}],$ and we have $:$%
\begin{equation*}
^{c}I\text{ }_{q_{1}}^{\delta }f\text{ }(q_{1})=0
\end{equation*}
\end{remark}

\begin{definition}
Let $f$ $:[q_{1},q_{2}]\rightarrow 
\mathbb{C}
$ be an absolutely continuous function on $[q_{1},q_{2}],$ then the Caputo
fractional derivative of $f$ of order $\delta $ and lower bound $q_{1}$ (%
\cite{DIE1},\cite{DIE2},\cite{POD},\cite{ROS1},\cite{ROS2}) is the function
denoted by $^{c}D$ $_{q_{1}}^{\delta }f$ and defined by $:$ 
\begin{equation*}
^{c}D\text{ }_{q_{1}}^{\delta }f\left( t\right) :=\frac{1}{\Gamma (1-\delta )%
}\int_{q_{1}}^{t}(t-s)^{-\delta }f^{\prime }(s)ds,\quad t\in \lbrack
q_{1},q_{2}]\ 
\end{equation*}

\begin{remark}
Let $f\in C^{1}([q_{1},q_{2}]).$ We have for every $x\in \lbrack
q_{1},q_{2}] $ $:$%
\begin{equation*}
\left( ^{c}I\text{ }_{q_{1}}^{\delta }\circ ^{c}D\text{ }_{q_{1}}^{\delta
}\right) f\left( x\right) =f\left( x\right) -f\left( q_{1}\right)
\end{equation*}%
If $f\left( q_{1}\right) =0,$ then the Caputo fractional integral of the
function $f$ of order $\delta ,$ $^{c}I$ $_{q_{1}}^{\delta }f$ \ is also of
class $C^{1}$ on the interval $[q_{1},q_{2}]$ and we have (\cite{DIE1},\cite%
{DIE2},\cite{POD},\cite{ROS1},\cite{ROS2})$:$ 
\begin{equation*}
\left( ^{c}D\text{ }_{q_{1}}^{\delta }\circ ^{c}I\text{ }_{q_{1}}^{\delta
}\right) f=f\text{ }
\end{equation*}
\end{remark}
\end{definition}

\subsection{Gevrey classes}

\begin{definition}
Let $l>0$. The Gevrey class of index $l$ on $\left[ q_{1},q_{2}\right] $,
denoted by $G_{l}(\left[ q_{1},q_{2}\right] ),$ is the set of all functions $%
f$ of class $C^{\infty }$ on $\left[ q_{1},q_{2}\right] $ such that $:$ 
\begin{equation*}
\left \Vert f^{\left( n\right) }\right \Vert _{\infty ,\left[ q_{1},q_{2}%
\right] }\leq B^{n+1}n^{n\left( 1+\frac{1}{l}\right) },\text{ }n\in 
\mathbb{N}%
\end{equation*}%
where $B>0$ is a constant $($with the convention that $0^{0}=1).$
\end{definition}

\begin{definition}
The Gevrey class of bound $q_{1}$ and index $l$ on the interval $\left[
q_{1},q_{2}\right] $, denoted by $G_{l,q_{1}}(\left[ q_{1},q_{2}\right] ),$
is the set of all functions $f$ of class $C^{1}$ on $\left[ q_{1},q_{2}%
\right] $ and of class $C^{\infty }$ on $]q_{1},q_{2}]$ such that the
restriction $f_{|\left[ q,q_{2}\right] }$ of $f$ belongs to the Gevrey class 
$G_{l}(\left[ q,q_{2}\right] )$, for every $q$ $\in \left] q_{1},q_{2}\right[
.$
\end{definition}

\subsection{The property $\mathcal{S}\left( l\right) $}

\begin{definition}
\textit{A function }$\varphi $\textit{\ defined on the set }$\{q_{1}\} \cup %
\left[ q_{1},q_{2}\right] ^{r}\;(r\in ]0,\pi \lbrack )$\textit{\ is said to
satisfy the \ property }$\mathcal{S}(l)$\textit{\ on the interval }$\left[
q_{1},q_{2}\right] $\textit{\ if }$\varphi _{|\left[ q_{1},q_{2}\right]
^{r}} $\textit{\ is holomorphic on }$\left[ q_{1},q_{2}\right] ^{r},\varphi
_{|\left[ q_{1},q_{2}\right] }$\textit{\ is a function of class }$C^{1}$%
\textit{\ on }$\left[ q_{1},q_{2}\right] $\textit{\ and there exists a
constant }$\tau _{\varphi }\in ]0,\pi \lbrack $\textit{\ such that for all }$%
D\in ]0,\tau _{\varphi }]$\textit{\ there exist }$N_{l,\varphi }(D)\in 
\mathbb{N}
^{\ast }$\textit{\ depending only on }$D,l$ and $\varphi $\textit{\ such
that the following inclusion }$:$%
\begin{equation}
\varphi (\left[ q_{1},q_{2}\right] ^{l,D,n+1})\subset \left[ q_{1},q_{2}%
\right] ^{l,D,n}  \label{EK}
\end{equation}%
\textit{holds for every integer }$n\geq N_{l,\varphi }(D).$\textit{\ The
number }$\tau _{\varphi }$\textit{\ is then called a }$\mathcal{S}(l)$-%
\textit{threshold for the function }$\varphi .$

\begin{remark}
Let $\varphi $\ be a function verifying the property $\mathcal{S}(l)$. Then $%
:$ 
\begin{equation*}
\varphi (\left[ q_{1},q_{2}\right] )\subset \left[ q_{1},q_{2}\right]
\end{equation*}%
On the other hand, it follows from $(\ref{EK})$ that we have for every $D\in
]0,\tau _{\varphi }[$ $:$%
\begin{equation*}
\varphi (\left[ q_{1},q_{2}\right] ^{l,D,N_{l,\varphi }(D)\left( p+1\right)
})\subset \left[ q_{1},q_{2}\right] ^{l,A,N_{l,\varphi }(D)p},\text{ }p\in 
\mathbb{N}
^{\ast }
\end{equation*}%
Thence we have $:$%
\begin{equation*}
\varphi (\left[ q_{1},q_{2}\right] ^{l,DN_{l,\varphi }(D)^{-\frac{1}{l}%
},p+1})\subset \left[ q_{1},q_{2}\right] ^{l,DN_{l,\varphi }(D)^{-\frac{1}{l}%
},p},\text{ }p\in 
\mathbb{N}
^{\ast }
\end{equation*}%
It follows that for every $D\in ]0,\tau _{\varphi }[$\ there exists $E\in
]0,D[$\ such that $:$%
\begin{equation*}
\varphi (\left[ q_{1},q_{2}\right] ^{l,E,p+1})\subset \left[ q_{1},q_{2}%
\right] ^{l,E,p},\text{ }p\in 
\mathbb{N}
^{\ast }
\end{equation*}
\end{remark}
\end{definition}

\subsection{Statement of the main result$\ $}

Our main result in this paper is the following.

\begin{theorem}
\end{theorem}

\textit{Let} \textit{be given }$\lambda \in 
\mathbb{C}
$ \textit{and} $\sigma >0$\textit{. Let }$a,$ $b$ \textit{and }$\psi $%
\textit{\ be a holomorphic functions on }$[-1,1]_{\sigma }$\textit{\ and }$%
\Phi $\textit{\ an entire function. We assume\ that the function }$a$ 
\textit{is not identically vanishing, that there exist a constants }$\alpha
_{0},$\textit{\ }$\beta _{0}>0$\textit{\ such that }$:$\textit{\  \ }%
\begin{equation*}
|\Phi (z)|\leq \alpha _{0}e\text{ }^{\beta _{0}|z|},\text{ }z\in 
\mathbb{C}%
\end{equation*}%
\textit{and that }$\psi $ \textit{satisfies the property }$\mathcal{S}(k)$%
\textit{. We assume also that the following conditions are fullfiled }$:$%
\textit{\ }%
\begin{equation}
a(-1)=\mathit{\ }b(-1)=0  \label{a}
\end{equation}

\begin{equation}
\frac{2^{\alpha }}{\alpha \Gamma (\alpha )}||b||_{_{\infty
,[-1,1]}}+|\lambda |\mathit{\ }<\frac{\ln \left( \frac{\alpha \Gamma (\alpha
)}{e\alpha _{0}\beta _{0}2^{\alpha }\Vert a\Vert _{_{\infty ,[-1,1]}}}%
\right) }{\beta _{0}}  \label{b}
\end{equation}

\begin{equation}
||a||_{\infty ,[-1,1]}||\Phi ||_{\infty ,\left[ -\frac{\ln \text{ }\left( 
\frac{\alpha \Gamma (\alpha )}{\alpha _{0}\beta _{0}2^{\alpha }\Vert a\Vert
_{\infty ,[-1,1]}}\right) }{\beta _{0}},\frac{\ln \text{ }\left( \frac{%
\alpha \Gamma (\alpha )}{\alpha _{0}\beta _{0}2^{\alpha }\Vert a\Vert
_{\infty ,[-1,1]}}\right) }{\beta _{0}}\right] }+||b||_{\infty ,[-1,1]}<1
\label{c}
\end{equation}%
\begin{equation}
||a||_{\infty ,[-1,1]}||\Phi ^{\prime }||_{\infty ,\left[ -\frac{\ln \text{ }%
(\frac{\alpha \Gamma (\alpha )}{\alpha _{0}\beta _{0}2^{\alpha }\Vert a\Vert
_{\infty ,[-1,1]}})}{\beta _{0}},\frac{\ln \text{ }(\frac{\alpha \Gamma
(\alpha )}{\alpha _{0}\beta _{0}2^{\alpha }\Vert a\Vert _{\infty ,[-1,1]}})}{%
\beta _{0}}\right] }<\frac{\alpha \Gamma (\alpha )}{2^{\alpha }}  \label{d}
\end{equation}

\begin{equation}
||\psi ^{\prime }||_{\infty ,[-1,1]}\leq 1+\alpha  \label{e}
\end{equation}%
\textit{\ Then the FFDE }$:$%
\begin{equation*}
(E):\text{ }^{c}D_{-1}^{\alpha }y(t)=a(t)\Phi (y\text{ }(\psi (t\text{ }%
)))+b(t)\text{ \  \ }
\end{equation*}%
\textit{has a solution }$u$\textit{\ which belongs to the Gevrey class }$%
G_{k,-1}(\left[ -1,1\right] )$\textit{\ and verifies the initial condition }$%
:$%
\begin{equation*}
(E_{1}):y(-1\text{ })=\lambda \text{ \  \  \  \ }
\end{equation*}

\section{\textbf{Proof of the main result}}

The proof of the theorem is subdivided in three steps.

\begin{itemize}
\item \textbf{Step 1 : The localisation of the solutions of the equation : }%
\begin{equation*}
(\Im ):r=\frac{\alpha _{0}2^{\alpha }}{\alpha \Gamma (\alpha )}\Vert a\Vert
_{\infty ,[-1,1]}e^{\beta _{0}r}+\frac{2^{\alpha }}{\alpha \Gamma (\alpha )}%
\Vert b\Vert _{\infty ,[-1,1]}+|\lambda |
\end{equation*}
\end{itemize}

The study of the variations of the function $:$%
\begin{equation*}
H:t\mapsto \frac{\alpha _{0}2^{\alpha }}{\alpha \Gamma (\alpha )}\Vert
a\Vert _{\infty ,[-1,1]}e^{\beta _{0}t}+\frac{2^{\alpha }}{\alpha \Gamma
(\alpha )}\Vert b\Vert _{\infty ,[-1,1]}+|\lambda |-t
\end{equation*}%
\ shows, under the condition (\ref{b})$,$ that $H$\ is strictly decreasing
on $\left[ 0,\frac{\ln \text{ }\left( \frac{\alpha \Gamma (\alpha )}{\alpha
_{0}\beta _{0}2^{\alpha }\Vert a\Vert _{\infty ,[-1,1]}}\right) }{\beta _{0}}%
\right] $ and strictly increasing on $\left[ \frac{\ln \text{ }\left( \frac{%
\alpha \Gamma (\alpha )}{\alpha _{0}\beta _{0}2^{\alpha }\Vert a\Vert
_{\infty ,[-1,1]}}\right) }{\beta _{0}},+\infty \right[ .$ But : 
\begin{equation*}
H(0)=\frac{\alpha _{0}2^{\alpha }}{\alpha \Gamma (\alpha )}\Vert a\Vert
_{\infty ,[-1,1]}+\frac{2^{\alpha }}{\alpha \Gamma (\alpha )}\Vert b\Vert
_{\infty ,[-1,1]}+|\lambda |>0
\end{equation*}%
and 
\begin{eqnarray*}
&&H\left( \frac{\ln \text{ }\left( \frac{\alpha \Gamma (\alpha )}{\alpha
_{0}\beta _{0}2^{\alpha }\Vert a\Vert _{\infty ,[-1,1]}}\right) }{\beta _{0}}%
\right) \\
&=&\frac{2^{\alpha }}{\alpha \Gamma (\alpha )}||b||_{_{\infty
,[-1,1]}}+|\lambda |\mathit{\ }-\frac{\ln \left( \frac{\alpha \Gamma (\alpha
)}{e\alpha _{0}\beta _{0}2^{\alpha }\Vert a\Vert _{_{\infty ,[-1,1]}}}%
\right) }{\beta _{0}}<0
\end{eqnarray*}%
Therefore the equation $(\Im )$ has on $%
\mathbb{R}
^{+}$ exactly two solutions $R_{0}<R_{1}$ and the following inequalities
hold $:$%
\begin{equation*}
0<R_{0}<\frac{\ln \text{ }\left( \frac{\alpha \Gamma (\alpha )}{\alpha
_{0}\beta _{0}2^{\alpha }\Vert a\Vert _{\infty ,[-1,1]}}\right) }{\beta _{0}}%
<R_{1}
\end{equation*}

\begin{itemize}
\item \textbf{Step 2 : Proof of the existence of a solution }$u$\textbf{\ of
the FFDE }$(E)$\textbf{\ in }$C^{1}([-1,1])$\textbf{\ such that the initial
condition }$(E_{1})$\textbf{\ holds.}
\end{itemize}

Consider the operator $T:C^{0}([-1,1])\longrightarrow C^{0}([-1,1])$ defined
by the formula $:$%
\begin{equation*}
T\left( f\right) \left( t\right) =^{c}I_{-1}^{\alpha }\left( a.(\Phi \circ
f\circ \psi )+b\right) \left( t\right) +\lambda ,\text{ }t\in \lbrack -1,1]
\end{equation*}%
We have for all $f\in \overline{\Delta }_{\infty }(R_{0}):$%
\begin{eqnarray*}
&&\Vert T(f\text{ })\Vert _{\infty ,[-1,1]} \\
&\leq &\Vert ^{c}I\text{ }_{-1}^{\alpha }\left( a.(\Phi \circ f\text{ }\circ
\psi )+b\right) \Vert _{\infty ,[-1,1]}+|\lambda \text{ }| \\
&\leq &\frac{2^{\alpha }}{\alpha \Gamma (\alpha )}\Vert a\Vert _{\infty
,[-1,1]}\Vert \Phi \circ f\text{ }\Vert _{\infty ,[-1,1]}+\frac{2^{\alpha }}{%
\alpha \Gamma (\alpha )}\Vert b\Vert _{\infty ,[-1,1]}+|\lambda \text{ }| \\
&\leq &\frac{2^{\alpha }}{\alpha \Gamma (\alpha )}\Vert a\Vert _{\infty
,[-1,1]}\alpha _{0}e^{\beta _{0}||f\text{ }\Vert _{\infty ,[-1,1]}}+\frac{%
2^{\alpha }}{\alpha \Gamma (\alpha )}\Vert b\Vert _{\infty ,[-1,1]}+|\lambda 
\text{ }| \\
&\leq &\frac{\alpha _{0}2^{\alpha }}{\alpha \Gamma (\alpha )}\Vert a\Vert
_{\infty ,[-1,1]}e^{\beta R_{0}}+\frac{2^{\alpha }}{\alpha \Gamma (\alpha )}%
\Vert b\Vert _{\infty ,[-1,1]}+|\lambda \text{ }|=R_{0}
\end{eqnarray*}%
Thence the closed ball $\overline{\Delta }_{\infty }(R_{0})$ is stable by
the operator $T.$ On the other hand, we have for all $f$ $,$ $g\in $ $%
\overline{\Delta }_{\infty }(R_{0}):$ 
\begin{eqnarray*}
&&||T(f)-T(g)||_{\infty ,[-1,1]} \\
&\leq &\frac{2^{\alpha }}{\alpha \Gamma (\alpha )}||a||_{\infty
,[-1,1]}||\Phi ^{\prime }||_{\infty ,[-R_{0},R_{0}]}||f-g||_{\infty ,[-1,1]}%
\text{ }
\end{eqnarray*}%
Since $0<R_{0}<$ $\frac{\ln \text{ }\left( \frac{\alpha \Gamma (\alpha )}{%
\alpha _{0}\beta _{0}2^{\alpha }\Vert a\Vert _{\infty ,[-1,1]}}\right) }{%
\beta _{0}}$ it follows from the condition (\ref{d}) that :%
\begin{equation*}
\frac{2^{\alpha }}{\alpha \Gamma (\alpha )}||a||_{\infty ,[-1,1]}||\Phi
^{\prime }||_{\infty ,[-R_{0},R_{0}]}<1
\end{equation*}%
Thence $T$ has, in $\overline{\Delta }_{\infty }(R_{0}),$ a unique fixed
point $u.$

Consider the sequence of functions $(f_{n})_{n\in 
\mathbb{N}
}$ defined on $\left[ -1,1\right] $ by the formula : 
\begin{equation*}
f_{n}:=T^{<n>}(f_{0}),\text{ }n\in 
\mathbb{N}%
\end{equation*}%
where $f_{0}$ is the null function. Direct computations show that the
functions $f_{n}$ belong to $\overline{\Delta }_{\infty }(R_{0}),$ are of
class $C^{1}$ on $[-1,1]$ and verify the inequality :%
\begin{equation*}
||f_{n+1}-f_{n}||_{_{\infty ,[-1,1]}}\leq ||f_{1}||_{_{\infty ,[-1,1]}}Q^{n},%
\text{ }n\in 
\mathbb{N}
\quad
\end{equation*}%
where :%
\begin{equation*}
Q:=\frac{2^{\alpha }}{\alpha \Gamma (\alpha )}||a||_{\infty ,[-1,1]}||\Phi
^{\prime }||_{\infty ,[-R_{0},R_{0}]}
\end{equation*}%
Let us set for each $n\in 
\mathbb{N}
,$ $F_{n}:=f_{n+1}-f_{n}.$ Thence since $Q\in \left[ 0,1\right[ ,$ it
follows that the function series $\sum F_{n}$ is uniformly convergent on $%
[-1,1]$ to a function $v\in \overline{\Delta }_{\infty }(R_{0})$ which is a
fixed point of the operator $T.$ It follows that $v=u.$ Consequently the
function series $\sum F_{n}$ is uniformly convergent of $[-1,1]$ to the
function $u_{\text{ }}$ $\in C^{0}([-1,1]).$ On the other hand we have for
all $x\in ]-1,1]$ and $n\in 
\mathbb{N}
^{\ast }:$%
\begin{eqnarray*}
&&F_{n+1}^{\prime }(x) \\
&=&\frac{\alpha (x+1)^{\alpha -1}}{\Gamma (\alpha )}\underset{0}{\overset{1}{%
\int }}(1-t)^{\alpha -1}a(-1+t(x+1))\cdot \\
&&\cdot \left( 
\begin{array}{c}
\Phi (f_{n}(\psi (-1+t(x+1))))- \\ 
-\Phi (f_{n-1}(\psi (-1+t(x+1))))%
\end{array}%
\right) dt+ \\
&&+\frac{(x+1)^{\alpha }}{\Gamma (\alpha )}\underset{0}{\overset{1}{\int }}%
(1-t)^{\alpha -1}ta^{\prime }(-1+t(x+1))\cdot \\
&&\cdot \left( 
\begin{array}{c}
\Phi (f_{n}(\psi (-1+t(x+1))))- \\ 
-\Phi (f_{n-1}(\psi (-1+t(x+1))))%
\end{array}%
\right) dt+ \\
&&+\frac{(x+1)^{\alpha }}{\Gamma (\alpha )}\underset{0}{\overset{1}{\int }}%
(1-t)^{\alpha -1}ta(-1+t(x+1))\psi ^{\prime }(-1+t(x+1))\cdot \\
&&\cdot \left( 
\begin{array}{c}
\left( \Phi ^{\prime }(f_{n}(\psi (-1+t(x+1))))-\Phi ^{\prime }(f_{n-1}(\psi
(-1+t(x+1))))\right) \cdot \\ 
\cdot f_{n}^{\prime }(\psi (-1+t(x+1))) \\ 
+\Phi ^{\prime }(f_{n-1}(\psi (-1+t(x+1))))\cdot \\ 
\cdot \left( f_{n}^{\prime }(\psi (-1+t(x+1)))-f_{n-1}^{\prime }(\psi
(-1+t(x+1)))\right)%
\end{array}%
\right) dt
\end{eqnarray*}%
Since $a(-1)=0$ it follows that :%
\begin{eqnarray*}
&&||F_{n+1}^{\prime }||_{\infty ,[-1,1]} \\
&\leq &\frac{2^{\alpha }}{(\alpha +1)\Gamma (\alpha )}||a^{\prime
}||_{\infty ,[-1,1]}||\Phi ^{\prime }||_{\infty
,[-R_{0},R_{0}]}||F_{n}||_{\infty ,[-1,1]}+ \\
&&+\frac{2^{\alpha }}{\alpha (\alpha +1)\Gamma (\alpha )}||a^{\prime
}||_{\infty ,[-1,1]}||\Phi ^{\prime }||_{\infty
,[-R_{0},R_{0}]}||F_{n}||_{\infty ,[-1,1]}+ \\
&&+\frac{2^{\alpha }}{\alpha (\alpha +1)\Gamma (\alpha )}||\psi ^{\prime
}||_{\infty ,[-1,1]}||a^{\prime }||_{\infty ,[-1,1]}||\Phi ^{^{\prime \prime
}}||_{\infty ,[-R_{0},R_{0}]}||f_{n}^{\prime }||_{\infty
,[-1,1]}||F_{n}||_{\infty ,[-1,1]}+ \\
&&+\frac{2^{\alpha }}{\alpha (\alpha +1)\Gamma (\alpha )}||\psi ^{\prime
}||_{\infty ,[-1,1]}||a^{\prime }||_{\infty ,[|1,1]}||\Phi ^{\prime
}||_{\infty ,[-R_{0},R_{0}]}||F_{n}^{\prime }||_{\infty ,[-1,1]} \\
&\leq &\frac{2^{\alpha }}{\alpha \Gamma (\alpha )}||a^{\prime }||_{\infty
,[-1,1]}||\Phi ^{\prime }||_{\infty ,[-R_{0},R_{0}]}||f_{1}||_{_{\infty
,[-1,1]}}Q^{n}+ \\
&&+\frac{2^{\alpha }}{\alpha (\alpha +1)\Gamma (\alpha )}||\psi ^{\prime
}||_{\infty ,[-1,1]}||a||_{\infty ,[-1,1]}||\Phi ^{^{\prime \prime
}}||_{\infty ,[-R_{0},R_{0}]}||f_{n}^{\prime }||_{\infty
,[-1,1]}||f_{1}||_{_{\infty ,[-1,1]}}Q^{n}+ \\
&&+\frac{2^{\alpha }}{\alpha \Gamma (\alpha )}||a||_{\infty ,[-1,1]}||\Phi
^{\prime }||_{\infty ,[-R_{0},R_{0}]}||F_{n}^{\prime }||_{\infty ,[-1,1]}
\end{eqnarray*}

\bigskip

\newpage

\bigskip

To achieve the proof of this step we need the following result.

\begin{prop}
\textit{The sequence }$(||f_{n}^{\prime }||_{_{\infty ,[-1,1]}})_{n\in 
\mathbb{N}
^{\ast }}$\textit{\ is bounded.}
\end{prop}

\begin{proof}
We have for all $x\in ]-1,1]$ and $n\in 
\mathbb{N}
^{\ast }:$ 
\begin{eqnarray*}
f_{n+1}^{\prime }(x) &=&\frac{\alpha (x+1)^{\alpha -1}}{\Gamma (\alpha )}%
\underset{0}{\overset{1}{\int }}(1-s)^{\alpha -1}\cdot \\
&&\cdot \left( 
\begin{array}{c}
a(-1+s(x+1))\cdot \\ 
\cdot \Phi (f_{n}(\psi (-1+s(x+1))))+ \\ 
+b(-1+s(x+1))%
\end{array}%
\right) ds+ \\
&&+\frac{(x+1)^{\alpha }}{\Gamma (\alpha )}\underset{0}{\overset{1}{\int }}%
(1-s)^{\alpha -1}s\cdot \\
&&\cdot \left( 
\begin{array}{c}
a^{\prime }(-1+s(x+1))\cdot \\ 
\cdot \Phi (f_{n}(\psi (-1+s(x+1))))+ \\ 
+b^{\prime }(-1+s(x+1)%
\end{array}%
\right) ds+ \\
&&+\frac{(x+1)^{\alpha }}{\Gamma (\alpha )}\underset{0}{\overset{1}{\int }}%
(1-s)^{\alpha -1}sa(-1+s(x+1))\cdot \\
&&\cdot \psi ^{\prime }(-1+s(x+1))\Phi ^{\prime }(f_{n}(\psi
(-1+s(x+1))))\cdot \\
&&\cdot f_{n}^{\prime }(\psi (-1+s(x+1)))ds
\end{eqnarray*}%
It follows from the assumption (\ref{a}) that :%
\begin{eqnarray*}
||f_{n+1}^{\prime }||_{\infty ,[-1,1]} &\leq &\frac{2^{\alpha }}{\alpha
\Gamma (\alpha )}(||a^{\prime }||_{\infty ,[-1,1]}||\Phi ^{\prime
}||_{\infty ,[-R_{0},R_{0}]}+ \\
&&+||b^{\prime }||_{\infty ,[-1,1]})+\frac{2^{\alpha }}{\alpha (\alpha
+1)\Gamma (\alpha )}||a||_{\infty ,[-1,1]}\cdot \\
&&\cdot ||\psi ^{\prime }||_{\infty ,[-1,1]}||\Phi ^{\prime }||_{\infty
,[-R_{0},R_{0}]}||f_{n}^{\prime }||_{\infty ,[-1,1]}
\end{eqnarray*}%
But we have according to the assumption (\ref{e}) that :%
\begin{eqnarray*}
&&\frac{2^{\alpha }}{\alpha (\alpha +1)\Gamma (\alpha )}||a||_{\infty
,[-1,1]}||\psi ^{\prime }||_{\infty ,[-1,1]}||\Phi ^{\prime }||_{\infty
,[-R_{0},R_{0}]} \\
&\leq &Q<1
\end{eqnarray*}%
Consequently the following inequality holds for each $n\in 
\mathbb{N}
^{\ast }$ :%
\begin{eqnarray*}
||f_{n+1}^{\prime }||_{\infty ,[-1,1]} &\leq &\frac{2^{\alpha }}{\alpha
\Gamma (\alpha )}(||a^{\prime }||_{\infty ,[-1,1]}||\Phi ^{\prime
}||_{\infty ,[-R_{0},R_{0}]}+ \\
&&+||b^{\prime }||_{\infty ,[-1,1]})+Q||f_{n}^{\prime }||_{\infty ,[-1,1]}
\end{eqnarray*}%
Since $Q\in \lbrack 0,1[,$ it follows that the sequence $(||f_{n}^{\prime
}||_{_{\infty ,[-1,1]}})_{n\in 
\mathbb{N}
^{\ast }}$\ is bounded.

The proof of the proposition is complete.
\end{proof}

Now we set :%
\begin{equation*}
\theta :=||f_{1}||_{_{\infty ,[-1,1]}}\left( 
\begin{array}{c}
\frac{2^{\alpha }}{\alpha \Gamma (\alpha )}||a||_{\infty ,[-1,1]}||\Phi
^{\prime }||_{\infty ,[-R_{0},R_{0}]}+ \\ 
+\frac{2^{\alpha }}{\alpha (\alpha +1)\Gamma (\alpha )}||\psi ^{\prime
}||_{\infty ,[-1,1]}||a||_{\infty ,[|1,1]}||\Phi ^{^{\prime \prime
}}||_{\infty ,[-R_{0},R_{0}]}\cdot \\ 
\cdot \underset{n\in 
\mathbb{N}
^{\ast }}{\sup }||f_{n}^{\prime }||_{_{\infty ,[-1,1]}}%
\end{array}%
\right)
\end{equation*}%
Then we can write :%
\begin{equation*}
||F_{n+1}^{\prime }||_{\infty ,[-1,1]}\leq \theta Q^{n}+Q||F_{n}^{\prime
}||_{\infty ,[-1,1]},\text{ }n\in 
\mathbb{N}
^{\ast }
\end{equation*}%
Direct computations show then that :%
\begin{equation*}
||F_{n+1}^{\prime }||_{\infty ,[-1,1]}\leq \theta nQ^{n}+||F_{1}^{\prime
}||_{\infty ,[-1,1]}Q^{n},\text{ }n\in 
\mathbb{N}
^{\ast }
\end{equation*}%
Since $Q\in \lbrack 0,1[,$\ it follows that the function series $\sum
F_{n}^{\prime }$ is uniformly convergent on $[-1,1].$ Thence the functions $%
u $ is of class $C^{1}$ on $[-1,1]$ and satisfies the relation :%
\begin{equation*}
^{c}I_{-1}^{\alpha }\left( a.(\Phi \circ u\circ \psi )+b\right) +\lambda =u
\end{equation*}%
Consequently, according to the assumption (\ref{a}), we can write for all $%
t\in \left[ -1,1\right] :$%
\begin{eqnarray*}
^{c}D\text{ }_{-1}^{\alpha }u\left( t\right) &=&^{c}D_{-1}^{\alpha
}[^{c}I_{-1}^{\alpha }\left( a.(\Phi \circ u\circ \psi )+b\right) ]\left(
t\right) \\
&=&a\left( t\right) (\Phi \left( u\left( \psi \left( t\right) \right)
\right) )+b\left( t\right)
\end{eqnarray*}%
So $u$ is a solution of the FFDE $\left( E\right) $ which belongs to $%
C^{1}([-1,1])$ and fullfiles the relation $u(-1)=\lambda .$

\begin{itemize}
\item \textbf{Step 3 : Proof that }$u$\textbf{\ belongs to the Gevrey class }%
$G_{k,-1}(\left[ -1,1\right] ).$
\end{itemize}

Since the function $\Lambda $ defined on $[0,\min (1,\sigma )[$ by :%
\begin{equation*}
\begin{array}{cccc}
\Lambda : & [0,\min (1,\sigma )[ & \rightarrow & 
\mathbb{R}
\\ 
& 0 & \mapsto & \frac{2^{\alpha }}{\alpha \Gamma (\alpha )}\max \left( 
\begin{array}{c}
\Vert a\Vert _{\infty ,[-1,1]}||\Phi ^{\prime }||_{\infty ,[-R_{0},R_{0}]},
\\ 
||a||_{\infty ,[-1,1]}||\Phi ||_{\infty ,[-R_{0},R_{0}]}+||b||_{\infty
,[-1,1]}%
\end{array}%
\right) \\ 
& s>0 & \mapsto & \frac{\left( 2+s\right) ^{\alpha }}{\alpha \Gamma (\alpha )%
}\max \left( 
\begin{array}{c}
\Vert a\Vert _{\infty ,[-1,1]^{s}}||\Phi ^{\prime }||_{\infty
,[-R_{0},R_{0}]_{s}}, \\ 
||a||_{\infty ,[-1,1]^{s}}||\Phi ||_{\infty
,[-R_{0},R_{0}]_{s}}+||b||_{\infty ,[-1,1]^{s}}%
\end{array}%
\right)%
\end{array}%
\end{equation*}%
\ is continuous on $[0,\min (1,\sigma )[$ and verifies by virtue of the
assumptions$\ $(\ref{c}) and (\ref{d}) the inequality $\Lambda (0)<1.$ It
follows that there exists $s_{1}\in ]0,\min (1,\sigma ,\tau _{\psi })[$ such
that :%
\begin{equation}
\text{ \ }\Lambda ([0,s_{1}])\subset \lbrack 0,1[  \label{contract1}
\end{equation}%
where $\tau _{\psi }$ is a $\mathcal{S}\left( k\right) $-threshold of $\psi
. $ Let $d$ be an arbitrary but fixed element of $\left] -1,1\right[ .$
Thanks to the remark $2.7.,$ there exists $s_{2}\in ]0,\min \left(
s_{1},1,1+d\right) [$ such that the functions $a$ and $b$ are both
holomorphic on $[-1,1]^{s_{2}}$ and the following condition holds :%
\begin{equation}
\text{ \ }\psi ([-1,1]^{k,s_{2},n+1})\subset \lbrack -1,1]^{k,s_{2,}n}\text{ 
},\text{ }n\in 
\mathbb{N}
^{\ast }  \label{psy}
\end{equation}%
Consider the sequence of functions $(\omega
_{n}:[-1,1]^{k,s_{_{2},}n}\rightarrow 
\mathbb{C}
)_{_{n\in 
\mathbb{N}
^{\ast }}}$ where : 
\begin{equation*}
\omega _{1}(z):=0,\text{ }z\in \lbrack -1,1]^{k,s_{2},1}
\end{equation*}%
and : 
\begin{eqnarray*}
&&\omega _{n+1}(z) \\
&:&=\frac{(z+1)^{\alpha }}{\Gamma (\alpha )}\int_{0}^{1}(1-s)^{\alpha
-1}\left( 
\begin{array}{c}
a(-1+s(z+1))\Phi (\omega _{n}(\psi (-1+s(z+1))))+ \\ 
+b(-1+s(z+1))%
\end{array}%
\right) ds+\lambda
\end{eqnarray*}%
for each $n\in 
\mathbb{N}
^{\ast }$and $z\in \lbrack -1,1]^{k,s_{2},n}.$ Then direct computations,
based on (\ref{psy}), show that the function $\omega _{n}$ is for every $%
n\in 
\mathbb{N}
^{\ast }$ holomorphic on $[-1,1]^{k,s_{2},n}.$

\bigskip

\bigskip \newpage

\bigskip

\bigskip

\begin{prop}
\textit{The inclusion } $\omega _{n}([-1,1]^{k,s_{2},n})\subset \lbrack
-R_{0},R_{0}]_{k,s_{2},n}$ \textit{holds for every}$\ n\in 
\mathbb{N}
^{\ast }.\  \  \  \ $ \ 
\end{prop}

\begin{proof}
We denote the last inclusion by $\mathcal{P}(n).$ We denote for every $z\in 
\mathbb{C}
$ by $\widehat{z}$ the closest point of $[-1,1]$ to $z.$ It is obvious that $%
\mathcal{P}(1)$ is true. Assume for a certain $n\in 
\mathbb{N}
^{\ast }$ that $\mathcal{P}(p)$ is true for every $p\in \left \{
1,...,n\right \} $ $.$ Since the function $\omega _{n+1}$ is holomorphic on $%
[-1,1]^{k,s_{2},n+1},$ we have then for each $z\in \lbrack
-1,1]^{k,s_{2},n+1}$ :%
\begin{eqnarray*}
&&\varrho (\omega _{n+1}(z),[-R_{0},R_{0}]) \\
&\leq &|\omega _{n+1}(z)-\omega _{n+1}(\widehat{z})| \\
&\leq &\frac{1}{\Gamma (\alpha )}\underset{\underrightarrow{\widehat{z},z}}{%
\int }\ |z-\zeta |^{\alpha -1}|a(\zeta )\Phi (\omega _{n}(\psi (\zeta
))|.|d\zeta |+ \\
&&+\frac{1}{\Gamma (\alpha )}\underset{\underrightarrow{\widehat{z},z}}{\int 
}|z-\zeta |^{\alpha -1}|b(\zeta )|.|d\zeta | \\
&\leq &\frac{\left( 2+s_{2}\right) ^{\alpha }}{\alpha \Gamma (\alpha )}%
\left( 
\begin{array}{c}
||a||_{\infty ,[-1,1]^{k,s_{2},n}}||\Phi ||_{\infty
,[-R_{0},R_{0}]_{k,s_{2},n}}+ \\ 
+||b||_{\infty ,[-1,1]^{k,s_{2},n}}%
\end{array}%
\right) \cdot \\
&&\cdot \varrho (z,[-1,1]) \\
&\leq &\Lambda (s_{2})s_{2}(n+1)^{-\frac{1}{k}} \\
&<&s_{2}(n+1)^{\frac{-1}{k}}
\end{eqnarray*}%
Thence the assertion $\mathcal{P}(n+1)$ is true. Consequently $\mathcal{P}%
(n) $ is true for all $n\in 
\mathbb{N}
^{\ast }.$

The proof of the proposition is then complete.
\end{proof}

By virtue of the proposition $3.2.$, we have for all $n\in 
\mathbb{N}
^{\ast }$ and $z\in \lbrack -1,1]^{k,s_{2},n+1}:$%
\begin{eqnarray*}
&&\left \vert \omega _{n+1}(z)-\omega _{n}(z)\right \vert \\
&\leq &\frac{\left \vert z+1\right \vert ^{\alpha }}{\Gamma (\alpha )}%
\int_{0}^{1}\left \vert 1-s\right \vert ^{\alpha -1}\left \vert
a(-1+s(z+1))\right \vert \cdot \\
&&\cdot \left \vert \Phi (\omega _{n}(\psi (-1+s(z+1))))-\Phi (\omega
_{n-1}(\psi (-1+s(z+1))))\right \vert ds \\
&\leq &\frac{\left( 2+s_{2}\right) ^{\alpha }}{\alpha \Gamma (\alpha )}\left
\Vert a\right \Vert _{\infty ,[-1,1]^{k,s_{5},n+1}}||\Phi ^{\prime
}||_{\infty ,[-R_{0},R_{0}]_{k,s_{2},n+1}}||\omega _{n}-\omega
_{n-1}||_{\infty ,[-1,1]^{k,s_{2},n}} \\
&\leq &\Lambda (s_{2})||\omega _{n}-\omega _{n-1}||_{\infty
,[-1,1]^{k,s_{2},n}}
\end{eqnarray*}%
It follows that $:$ 
\begin{equation*}
\left \Vert \omega _{n+1}-\left( \omega _{n|[-1,1]^{k,s_{2},n+1}}\right)
\right \Vert _{\infty ,[-1,1]^{k,s_{2},n+1}}\leq \frac{||\omega _{2}-\omega
_{1}||_{\infty ,[-1,1]^{k,s_{2},1}}}{\Lambda (s_{2})}\Lambda (s_{2})^{n},%
\text{ }n\in 
\mathbb{N}
^{\ast }
\end{equation*}%
Let us set $\Omega _{1}:=\omega _{1}$ and denote, for all $n\in 
\mathbb{N}
^{\ast }\backslash \left \{ 1\right \} ,$ by $\Omega _{n}$ the function :%
\begin{equation*}
\begin{array}{cccc}
\Omega _{n}: & [-1,1]^{k,s_{2},n\text{ }+1} & \rightarrow & 
\mathbb{C}
\\ 
& z & \mapsto & \omega _{n+1}\left( z\right) -\omega _{n}\left( z\right)%
\end{array}%
\end{equation*}%
Then the function $\Omega _{n}$ is holomorphic on $[-1,1]^{k,s_{2},n\text{ }%
+1}$ for each $n\in 
\mathbb{N}
^{\ast }.$ Furthermore the following relations hold for every $n\in 
\mathbb{N}
^{\ast }:$%
\begin{equation}
\left \{ 
\begin{array}{c}
||\Omega _{n}||_{\infty ,[-1,1]^{k,s_{2},n+1}}\leq \frac{||\omega
_{2}-\omega _{1}||_{\infty ,[-1,1]^{k,s_{2},1}}}{\Lambda (s_{2})}(\Lambda
(s_{2}))^{n} \\ 
\  \Omega _{n|[-1,1]}=f_{n-1}%
\end{array}%
\right.  \label{ASSUMPT}
\end{equation}%
Since $\Lambda (s_{2})\in \lbrack 0,1[,$ it follows then from (\ref{ASSUMPT}%
) that the function series $\sum \Omega _{n}{}_{|[-1,1]}$ is uniformly
convergent on $[-1,1]$ to the function $u.$ But we know, according to the
relation (\ref{inclusion}) of remark $2.1.,$ that the following inclusions
hold $:$%
\begin{equation*}
\lbrack d,1]_{_{k,s_{2},n}}\subset \lbrack -1,1]^{k,s_{2},n},n\in 
\mathbb{N}
^{\ast }
\end{equation*}%
it follows that :%
\begin{equation}
\left \{ 
\begin{array}{c}
||\Omega _{n}||_{\infty ,[d,1]_{_{k,s_{2},n}}}\leq \frac{||\omega
_{2}-\omega _{1}||_{\infty ,[-1,1]^{k,s_{2},1}}}{\Lambda (s_{2})}(\Lambda
(s_{2}))^{n},\text{ }n\in 
\mathbb{N}
^{\ast } \\ 
\  \Omega _{n|[d,1]}=f_{n-1|[d,1]},\text{ }n\in 
\mathbb{N}
^{\ast }%
\end{array}%
\right.  \label{ASSUMPT;NEW}
\end{equation}%
The relations (\ref{ASSUMPT;NEW}) entail, thanks to the main result of (\cite%
{BEND}), that $u_{|[d,1]}$ belongs to the Gevrey class $G_{k}([d,1]).$ But
since $d$ is an arbitrary element of $\  \left] -1,1\right[ $ and $u$ is of
class $C^{1}$ on $\left[ -1,1\right] ,$ it follows that $u$ belongs to the
Gevrey class\textbf{\ }$G_{k,-1}(\left[ -1,1\right] ).$

The proof of the main result is then complete.

\section{\textbf{Examples}}

To obtain examples illustrating our main result, we need first to prove the
following proposition.

\begin{prop}
\textit{The function }$:$ 
\begin{equation*}
\begin{array}{cccc}
\mathcal{L}: & 
\mathbb{C}
& \rightarrow & 
\mathbb{C}
\\ 
& z & \longmapsto & 2e^{\frac{z-1}{2}}-1%
\end{array}%
\end{equation*}

\textit{satisfies the property }$\mathcal{S}(l)$\textit{\ for every }$l\in
]0,1].$
\end{prop}

\begin{proof}
\bigskip Let $l\in ]0,1],$ $\varepsilon \in ]0,1]$ and $z\in \lbrack
-1,1]^{\varepsilon }.$ We have : 
\begin{equation*}
\mathcal{L}(z)+1=2e^{\frac{z-1}{2}}
\end{equation*}%
It follows that :%
\begin{equation*}
\func{Re}\left( \mathcal{L}(z)+1\right) >0
\end{equation*}%
We consider then the principal argument $\arg (\mathcal{L}(z)+1)$ of $%
\mathcal{L}(z)+1$ which satisfies the following estimates :%
\begin{equation*}
|\arg (\mathcal{L}(z)+1)|=\frac{|\func{Im}(z)|}{2}\leq \left( 1+\frac{%
\varepsilon }{2}\right) \tan \varepsilon
\end{equation*}%
But direct computations prove that :%
\begin{equation*}
0<\frac{\tan \varepsilon -\left( \varepsilon +\frac{\varepsilon ^{3}}{3}%
\right) }{\varepsilon ^{3}}<\tan 1-\frac{4}{3}
\end{equation*}%
Thence we have :%
\begin{eqnarray*}
\left( 1+\frac{\varepsilon }{2}\right) \tan \varepsilon &\leq &\left( 1+%
\frac{\varepsilon }{2}\right) \left( \varepsilon +\frac{\varepsilon ^{3}}{3}%
+\left( \tan 1-\frac{4}{3}\right) )\varepsilon ^{3}\right) \\
&\leq &\varepsilon +\left( \frac{3}{2}\tan 1-1\right) \varepsilon ^{2}
\end{eqnarray*}%
It follows that :%
\begin{equation}
\arg (\mathcal{L}(z)+1)\leq \varepsilon +\left( \frac{3}{2}\tan 1-1\right)
\varepsilon ^{2}  \label{one}
\end{equation}%
On the other hand we have :%
\begin{eqnarray*}
|\mathcal{L}(z)+1| &=&2e^{\frac{\func{Re}(z)-1}{2}} \\
&\leq &2e^{\frac{\varepsilon }{2}}
\end{eqnarray*}%
But we know that :%
\begin{eqnarray*}
2e^{\frac{\varepsilon }{2}} &=&2+\varepsilon +\varepsilon ^{2}\underset{n=0}{%
\overset{+\infty }{\sum }}\frac{\varepsilon ^{n}}{(n+2)!2^{n+2}} \\
&\leq &2+\varepsilon +\left( \sqrt{e}-\frac{3}{2}\right) \varepsilon ^{2}
\end{eqnarray*}%
It follows that :%
\begin{equation}
|\mathcal{L}(z)+1|\leq 2+\varepsilon +\left( \sqrt{e}-\frac{3}{2}\right)
\varepsilon ^{2}  \label{two}
\end{equation}%
We derive, from the estimates (\ref{one}) and (\ref{two}), the following
inclusion :%
\begin{equation}
\mathcal{L}([-1,1]^{\varepsilon })\subset \lbrack -1,1]^{\varepsilon +\nu
\varepsilon ^{2}}  \label{INCLUS}
\end{equation}%
where : 
\begin{eqnarray*}
\mu &:&=\max \left( \frac{3}{2}\tan 1-1,\sqrt{e}-\frac{3}{2}\right) \\
&=&\frac{3}{2}\tan 1-1>0
\end{eqnarray*}%
Let $n\in 
\mathbb{N}
^{\ast }$ and $A\in ]0,\frac{1}{\mu l}[.$ We have :%
\begin{eqnarray*}
&&A(n+1)^{\frac{-1}{l}}+\mu A^{2}(n+1)^{\frac{-2}{l}}-An^{\frac{-1}{l}} \\
&=&An^{\frac{-1}{l}}\left( (1+\frac{1}{n})^{\frac{-1}{l}}+\mu A\frac{n^{%
\frac{1}{l}}}{(n+1)^{\frac{2}{l}}}-1\right) \\
&\leq &An^{\frac{-1}{l}}\left( (1+\frac{1}{n})^{\frac{-1}{l}}+\frac{\mu A}{n}%
-1\right)
\end{eqnarray*}%
But we have :%
\begin{equation*}
\left( 1+\frac{1}{n}\right) ^{\frac{-1}{l}}+\frac{\mu A}{n}-1\underset{%
n\rightarrow +\infty }{\sim }\frac{\nu (A-\frac{1}{\mu l})}{n}
\end{equation*}%
It follows that there exists an integer $N_{A,l}\geq 1$ such that the
following inequality holds for every integer $n\geq N_{A,l}:$%
\begin{equation*}
A(n+1)^{\frac{-1}{l}}+\mu A^{2}(n+1)^{\frac{-2}{l}}\leq An^{\frac{-1}{l}}
\end{equation*}%
Consequently we have :%
\begin{equation*}
\mathcal{L}([-1,1]^{A(n+1)^{\frac{-1}{l}}})\subset \lbrack -1,1]^{An^{\frac{%
-1}{l}}},\text{ }n\geq N_{A,l}
\end{equation*}%
that is :%
\begin{equation*}
\mathcal{L}([-1,1]^{l,A,n+1})\subset \lbrack -1,1]^{l,A,n},\text{ }n\geq
N_{A,l}
\end{equation*}%
It follows that \text{the function satisfies the property\textit{\ }}%
\QTR{cal}{S}$(l).$

The proof of the proposition is then complete.
\end{proof}

\bigskip

\newpage

\bigskip

\begin{example}
\bigskip 
\end{example}

Let $C\in 
\mathbb{C}
$ and $\gamma \in ]-1,1[.$ We assume that :%
\begin{equation}
0<|C|<\min \left( \frac{\alpha \Gamma (\alpha )}{2^{\alpha +1}e^{\frac{%
2^{\alpha }}{\alpha \Gamma (\alpha )}|\gamma |+1}},\text{ }\frac{1-|\gamma |%
}{2}\right)  \notag
\end{equation}%
Consider the FFDE :%
\begin{equation*}
(E^{1}):\text{ }^{c}D_{-1}^{\alpha }f(x)=C(x+1)\sin \left( f(2e^{\frac{x-1}{2%
}}-1)\right) +\gamma \sin (x+1)
\end{equation*}%
with the initial condition :%
\begin{equation*}
(E_{1}^{1}):f(-1)=0
\end{equation*}%
Consider then the entire functions :%
\begin{equation*}
\left \{ 
\begin{array}{c}
\Phi _{1}:z\mapsto \sin (z) \\ 
a_{1}:z\mapsto C(x+1) \\ 
b_{1}:z\mapsto \gamma \sin (z+1)%
\end{array}%
\right.
\end{equation*}%
Thence $a_{3}$ is not identically vanishing and we have :%
\begin{equation*}
|\Phi _{1}(z)|\leq e^{|z|},\text{ }z\in 
\mathbb{C}%
\end{equation*}%
\textit{\ }So we can take $\alpha _{0}=\beta _{0}=1.$ Furthermore we have :%
\begin{eqnarray*}
\frac{2^{\alpha }}{\alpha \Gamma (\alpha )}||b_{1}||_{_{\infty ,[-1,1]}}+|0|%
\mathit{\ } &=&\frac{2^{\alpha }}{\alpha \Gamma (\alpha )}|\gamma | \\
&<&\ln \left( \frac{\alpha \Gamma (\alpha )}{e2^{\alpha +1}|C|}\right) \\
&=&\frac{\ln \left( \frac{\alpha \Gamma (\alpha )}{e\alpha _{0}\beta
_{0}2^{\alpha }\Vert a_{1}\Vert _{_{\infty ,[-1,1]}}}\right) }{\beta _{0}}
\end{eqnarray*}%
We have also :%
\begin{eqnarray*}
&&||a_{1}||_{\infty ,[-1,1]}||\Phi ^{\prime }||_{\infty ,\left[ -\frac{\ln 
\text{ }\left( \frac{\alpha \Gamma (\alpha )}{\alpha _{0}\beta _{0}2^{\alpha
}\Vert a_{1}\Vert _{\infty ,[-1,1]}}\right) }{\beta _{0}},\frac{\ln \text{ }%
\left( \frac{\alpha \Gamma (\alpha )}{\alpha _{0}\beta _{0}2^{\alpha }\Vert
a_{1}\Vert _{\infty ,[-1,1]}}\right) }{\beta _{0}}\right] } \\
&\leq &2|C|<\frac{\alpha \Gamma (\alpha )}{2^{\alpha }}
\end{eqnarray*}

\begin{eqnarray*}
&&||a_{1}||_{\infty ,[-1,1]}||\Phi ||_{\infty ,\left[ -\frac{\ln \text{ }%
\left( \frac{\alpha \Gamma (\alpha )}{\alpha _{0}\beta _{0}2^{\alpha }\Vert
a_{1}\Vert _{\infty ,[-1,1]}}\right) }{\beta _{0}},\frac{\ln \text{ }\left( 
\frac{\alpha \Gamma (\alpha )}{\alpha _{0}\beta _{0}2^{\alpha }\Vert
a_{1}\Vert _{\infty ,[-1,1]}}\right) }{\beta _{0}}\right] } \\
&&+||b_{1}||_{\infty ,[-1,1]} \\
&\leq &2|C|+|\gamma |<1
\end{eqnarray*}%
Consequently it follows from the main result that the problem $%
(E^{1})-(E_{1}^{1})$ has a solution which belongs to the Gevrey class $%
G_{1,-1}(\left[ -1,1\right] ).$

\bigskip

\newpage

\bigskip

\bigskip

\begin{example}
\end{example}

Let $\eta >0$ and $\lambda \in 
\mathbb{C}
.$ We assume that : 
\begin{equation}
\eta <\min \left( \frac{\alpha \Gamma (\alpha )}{2^{\alpha }e^{|\lambda |+1}}%
,\text{ }1\right)  \notag
\end{equation}%
Consider the FFDE :%
\begin{equation*}
(E^{2}):\text{ }^{c}D_{-1}^{\alpha }(x)=\eta \sin (x+1)\cos \left( f(2e^{%
\frac{x-1}{2}}-1)\right)
\end{equation*}%
with the initial condition :%
\begin{equation*}
(E_{1}^{2}):\text{\ }f(-1)=\lambda
\end{equation*}%
Consider then the functions :%
\begin{equation*}
\left \{ 
\begin{array}{c}
\Phi _{2}:z\mapsto \cos z \\ 
a_{2}:z\mapsto \eta \sin (z+1) \\ 
b_{2}:z\mapsto 0%
\end{array}%
\right.
\end{equation*}%
Thence we have :%
\begin{equation*}
\left \{ 
\begin{array}{c}
a_{2}([-1,1])\subset 
\mathbb{R}
^{\ast } \\ 
|\Phi _{2}(z)|\leq e^{|z|},\text{ }z\in 
\mathbb{C}%
\end{array}%
\right.
\end{equation*}%
So we can take $\alpha _{0}=\beta _{0}=1.$ Furthermore we have the following
inequalities :\textit{\ }$\  \  \  \  \ $%
\begin{eqnarray*}
\frac{2^{\alpha }}{\alpha \Gamma (\alpha )}||b_{2}||_{_{\infty
,[-1,1]}}+|\lambda | &=&\ln \left( \frac{\alpha \Gamma (\alpha )}{e2^{\alpha
}\eta }\right) \\
\mathit{\ } &<&\frac{\ln \left( \frac{\alpha \Gamma (\alpha )}{e\alpha
_{0}\beta _{0}2^{\alpha }\Vert a_{2}\Vert _{_{\infty ,[-1,1]}}}\right) }{%
\beta _{0}}
\end{eqnarray*}

\begin{eqnarray*}
||a_{2}||_{\infty ,[-1,1]}||\Phi _{2}^{\prime }||_{\infty ,\left[ -\frac{\ln 
\text{ }\left( \frac{\alpha \Gamma (\alpha )}{\alpha _{0}\beta _{0}2^{\alpha
}\Vert a_{2}\Vert _{\infty ,[-1,1]}}\right) }{\beta _{0}},\frac{\ln \text{ }%
\left( \frac{\alpha \Gamma (\alpha )}{\alpha _{0}\beta _{0}2^{\alpha }\Vert
a_{2}\Vert _{\infty ,[-1,1]}}\right) }{\beta _{0}}\right] } &\leq &\eta \\
&<&\frac{\alpha \Gamma (\alpha )}{2^{\alpha }}
\end{eqnarray*}

\begin{equation*}
||a_{2}||_{\infty ,[-1,1]}||\Phi _{2}||_{\infty ,\left[ -\frac{\ln \text{ }%
\left( \frac{\alpha \Gamma (\alpha )}{\alpha _{0}\beta _{0}2^{\alpha }\Vert
a_{2}\Vert _{\infty ,[-1,1]}}\right) }{\beta _{0}},\frac{\ln \text{ }\left( 
\frac{\alpha \Gamma (\alpha )}{\alpha _{0}\beta _{0}2^{\alpha }\Vert
a_{2}\Vert _{\infty ,[-1,1]}}\right) }{\beta _{0}}\right] }+||b_{2}||_{%
\infty ,[-1,1]}=\eta <1
\end{equation*}%
Consequently it follows from the main result that the problem $%
(E^{2})-(E_{1}^{2})$ has a solution which belongs to the Gevrey class $%
G_{1,-1}(\left[ -1,1\right] ).$

\smallskip

\bigskip

\bigskip

\end{document}